\documentclass[12pt, 14paper,reqno]{amsart}
\usepackage{amsmath,amsfonts,amssymb}

\theoremstyle{plain}
\newtheorem{theorem}{Theorem}
\newtheorem{lemma}{Lemma}
\newtheorem{remark}{Remark}

\theoremstyle{proof}
\theoremstyle{definition}

\numberwithin{equation}{section}
\numberwithin{lemma}{section}
\numberwithin{theorem}{section}

\usepackage{amsmath}
\usepackage{amsfonts}   
\usepackage{amssymb}
\usepackage{amssymb, amsmath, amsthm}
\usepackage[breaklinks]{hyperref}
\theoremstyle{thmrm}

\usepackage{graphicx}
\begin{document}
\title[Real quadratic fields with class number up to three]{A note on certain real quadratic fields with class number up to three}
\author{Kalyan Chakraborty, Azizul Hoque and Mohit Mishra}
\address{Kalyan Chakraborty @Kalyan Chakraborty, Harish-Chandra Research Institute, HBNI, Chhatnag Road, Jhunsi, Allahabad 211 019, India.}
\email{kalyan@hri.res.in}
\address{Azizul Hoque @Azizul Hoque, Harish-Chandra Research Institute, HBNI, Chhatnag Road, Jhunsi, Allahabad 211 019, India.}
\email{azizulhoque@hri.res.in}
\address{Mohit Mishra @Mohit Mishra, Harish-Chandra Research Institute, HBNI,
Chhatnag Road, Jhunsi,  Allahabad 211 019, India.}
\email{mohitmishra@hri.res.in}

\maketitle

\begin{abstract} We obtain criteria for the class number of certain Richaud-Degert type real quadratic fields to be $3$. We also treat a couple of families of real quadratic fields of Richaud-Degert type that were not considered earlier, and  obtain similar criteria for the class number of such fields to be $2$ and $3$. 
\end{abstract}

\section{Introduction}
The size of the class group of an algebraic number field is one of the fundamental problems in algebraic number theory. Gauss conjectured that there are exactly nine imaginary quadratic fields with class number $1$. This conjecture was proved independently by Baker \cite{BA66} and Stark \cite{ST67}. However, Heegner had already proved this conjecture in \cite{HE52}. Unfortunately, his
proof was regarded as incorrect or at the best, incomplete. Stark found that the gap in the proof is very minor and he had completed the same in \cite{ST69}. In fact, Gauss gave a list of imaginary quadratic fields with given very low class numbers, and he believed them to be complete. The list of imaginary quadratic fields with class number $2$ was completely classified by Baker and Stark  independently in \cite{BA71} and \cite{ST71} respectively, and jointly in \cite{BS71}. The analogous list of imaginary quadratic fields with class number $3$ was computed by Oesterl\'{e} in \cite{OS88}. Finally, M. Watkins \cite{WA04} classified all the imaginary quadratic fields with class numbers up to $100$.     

On the other hand, very little is known about the class number of real quadratic fields. In 1801, Gauss conjectured the following:
\begin{itemize}
\item[(G1)]There exist infinitely many real quadratic fields of class number
$1$, or more precisely
\item[(G2)] There exist infinitely many real quadratic fields of the form $\mathbb{Q}(\sqrt{p}),~ p\equiv 1\pmod4$ of class number $1$.
\end{itemize}
This conjecture is yet to be resolved. It seems that one of the most essential difficulties of this problem comes from deep connection of the class number with the fundamental unit. In connection to (G2), Chowla and Friedlander \cite{CF76} posted the following conjecture:
\begin{itemize}
\item[(CF)]\label{CF} If $D=m^2+1$ is a prime with $m>26$, then the class number of $\mathbb{Q}(\sqrt{D})$ is greater than $1$.
\end{itemize}
This conjecture says that there are exactly $7$ real quadratic fields of the form $\mathbb{Q}(\sqrt{m^2+1})$ with class number $1$, and they correspond to $m\in\{1, 2, 4, 6, 10, 14, 24\}$. In 1988, Mollin and Williams \cite{MW88}  proved this conjecture under the generalized Riemann hypothesis. Chowla also posted a conjecture analogous to (CF) on a general family of real quadratic fields. More precisely, he conjectured the following:   
\begin{itemize}
\item[(C)] Let $D$ be a square-free rational integer of the
form $D=4m^2+1$ for some positive integer $m$. Then there exist exactly $6$ real
quadratic fields $\mathbb{Q}(\sqrt{D})$ of class number one, viz.
$D\in \{5, 17, 37, 101, 197, 677\}$.
\end{itemize}
Yokoi \cite{YO86} studied this conjecture and he posted one more conjecture on another family of real quadratic fields. More precisely, he posted the following conjecture:
\begin{itemize}
\item[(Y)] Let $D$ be a square-free rational integer of the form
$D=m^2+4$ for some positive integer $m$. Then there exists exactly $6$ real quadratic fields $\mathbb{Q}(\sqrt{D})$ of class number one, viz.
$D\in\{5,13, 29, 53, 173, 293\}$.
\end{itemize}
Kim, Leu and Ono \cite{KLO87} proved that at least one of (C) and (Y) is true, and that there are at most $7$ real quadratic fields $\mathbb{Q}(\sqrt{D})$ of class number $1$ for the other case. The conjectures (C) and (Y) were proved by Bir\'{o} in \cite{B03, BI03}. Hoque and Saikia \cite{HS16} proved that there  do not exist any real quadratic fields of the form $\mathbb{Q}(\sqrt{9(8n^2+r)+2})$, where $n\geq 1$ and $r=5, 7$, with class number $1$. In \cite{CH17}, the authors proved that there are no real quadratic fields $\mathbb{Q}(\sqrt{d})$ of class number $1$ when $d=n^2p^2+1$ with $p\equiv \pm 1\pmod 8$ a prime and $n$ an odd integer. Recently, Hoque and Chakraborty \cite{CH18} proved that if $d$ is a square-free part of $an^2+2$, where $a=9, 196$ and $n$ is an odd integer, then the class number of $\mathbb{Q}(\sqrt{d})$ is greater than $1$.    
It is more interesting to find necessary and sufficient conditions that a real quadratic field has a given fixed class number $g$. Yokoi \cite{YO86} proved using algebraic method that for a positive integer $m$, the class number of $\mathbb{Q}(\sqrt{4m^2+1})$ is $1$ if and only if $m^2-t(t+1)$ is a prime for every $1\leq t\leq m-1$. Lu obtained this result in \cite{LU81} using the theory of continued fractions. Kobayashi \cite{KO90} obtained stronger conditions that this as well as some other families of real quadratic fields to be of class number $1$. In \cite{BK1}, Byeon and Kim established certain necessary and sufficient conditions for the class number of real quadratic fields of Richaud-Degert type to be $1$. They obtained in \cite{BK2} these conditions by comparing the special zeta values attached to a real quadratic field determined by two different ways of computation.       
Analogously, they also obtained some necessary and sufficient conditions for the class number of the real quadratic fields of Richaud-Degert type to be $2$. Mollin \cite{MO91} also obtained some analogous conditions for class number to be $2$ using the theory of continued fractions and algebraic arguments.  

In this paper, we consider all real quadratic fields of narrow Richaud-Degert type with two exceptions. More precisely, we consider the real quadratic fields $k=\mathbb{Q}(\sqrt{d})$, where $d=n^2+ r$ and $|r|\in\{1,4\}$ with the exceptions when $n^2-1\equiv 3\pmod 4$ and $n^2-4\equiv 5\pmod 8$. We also consider wide Richaud-Degert type real quadratic fields $\mathbb{Q}(\sqrt{d})$, where $d=n^2+r$ and $r\not\in\{1,4\}$ with $d\equiv 1\pmod 8$. We obtain some criteria for the class number of these fields to be $3$. We also obtain similar criteria for $\mathbb{Q}(\sqrt{n^2+r}), ~r\in\{1, 4\}$ to have class number $2$ which were not covered in \cite{BK2} by Byeon and Kim. We largely follow the method of Byeon and Kim \cite{BK1, BK2}. 

\section{Values of Dedekind zeta function}
In this section, we discuss two different ways of computing special values of zeta functions attached to a real quadratic field that are due to Siegel and Lang. Let $k$ be a real quadratic field, and $\zeta_{k}(s)$ be the Dedekind zeta function of $k$. By specializing Siegel's formula \cite{SI69} for $\zeta_k(1-2n)$ for general $k$, Zagier \cite{ZA76} described this formula by direct analytic methods when $k$ is a real quadratic field. For $n=1$, it takes the following form (see \S3 in \cite{ZA76}).
\begin{theorem}[Zagier \cite{ZA76}]\label{thm2.1}
Let $k$ be a real quadratic field with discriminant $D$. Then 
$$\zeta_k(-1)=\frac{1}{60}\sum_{\substack{ |t|<\sqrt{D}\\ t^2\equiv D\pmod 4}}\sigma\left(\frac{D-t^2}{4}\right),$$
where $\sigma(n)$ denotes the sum of divisors of $n$.\
\end{theorem}

Another method of computing special values of $\zeta_k(s)$ is due to Lang whenever $k$ is a real quadratic field. 
Let $k=\mathbb{Q}{(\sqrt{d})}$ be a real quadratic field with discriminant $D$, and let $\mathfrak{A}$ be an ideal class in $k$. Let $\mathfrak{a}$ be an integral ideal in $\mathfrak{A}^{-1}$ with an integral basis $\{r_{1},r_{2}\}$. We define $$\delta(\mathfrak{a})= r_1r_2'-r_1'r_2,$$
where $r_1'$ and  $r_2'$ are the conjugates of $r_1$ and $r_2$ respectively.

Let $\varepsilon$ be the fundamental unit of $k$. Then $\{\varepsilon r_1, \varepsilon r_2\}$ is also integral basis of $\mathfrak{a}$, and thus we can find a matrix 
$M=
\begin{bmatrix}
a&b\\
c&d
\end{bmatrix}
$
with integer entries satisfying 
$$\varepsilon\begin{bmatrix}
r_1\\r_2
\end{bmatrix}
=M\begin{bmatrix}
r_1\\
r_2
\end{bmatrix}.$$
We can now recall the following result of Lang \cite{LAN} which is one of the main ingredient to prove our results.
\begin{theorem}[Lang \cite{LAN}]\label{thm2.2}
By keeping the above notations, we have 
\begin{align*}
\zeta_k(-1, \mathfrak{A})&=\frac{\textsl{sgn }\delta(\mathfrak{a})~r_2r_2'}{360N(\mathfrak{a})c^3}\big\{(a+d)^3-6(a+d)N(\varepsilon)-240c^3(\textsl{sgn } c)\\
&\times S^3(a,c)+180ac^3(\textsl{sgn } c)S^2(a,c)-240c^3(\textsl{sgn } c)S^3(d,c)\\
& +180dc^3(\textsl{sgn } c)S^2(d,c) \big\},
\end{align*}
where $N(\mathfrak{a})$ is the norm of $\mathfrak{a}$ and $S^i(-,-)$ denotes the generalized Dedekind sum as defined in \cite{AP50}.
\end{theorem} 
We need to determine the values of $a,b,c,d$ and generalized Dedekind sum in order to apply Theorem \ref{thm2.2}. The following result of Kim \cite{KI88} helps us to determine the values of $a, b, c$ and $d$.
\begin{lemma}[Kim \cite{KI88}]\label{2.1} The entries of $M$ are given by 
\begin{align*}
a&=Tr\left(\frac{r_1 r_2'\varepsilon}{\delta(\mathfrak{a})}\right),~
b=Tr\left(\frac{r_1 r_1'\varepsilon'}{\delta(\mathfrak{a})}\right),~
c=Tr\left(\frac{r_2 r_2'\varepsilon}{\delta(\mathfrak{a})}\right) \text{ and }\\
d&=Tr\left(\frac{r_1 r_2'\varepsilon'}{\delta(\mathfrak{a})}\right).
\end{align*}
Moreover, $\det(M)=N(\varepsilon)$ and $bc\ne 0$.
\end{lemma} 
Kim \cite{KI88} obtained the following expressions for generalized Dedekind sum by using reciprocity law. These expressions are also needed to compute the values of zeta functions for ideal classes of the respective real quadratic fields.
\begin{lemma}[Kim \cite{KI88}]\label{DS1} For any positive integer $m$, we have
\begin{itemize}
\item[(i)] $S^3(\pm 1, m)=\pm(-m^4+5m^2-4)/(120m^3),$\vspace*{2mm}
\item[(ii)] $S^2(\pm 1, m)=(m^4+10m^2-6)/(180m^3).$
\end{itemize}
\end{lemma}
\begin{lemma}[Kim \cite{KI88}]\label{DS2} For any positive even integer $m$, we have
\begin{itemize}
\item[(i)] $S^3(m\pm 1, 2m)=\pm S^1(m+1, 2m)=\mp(m^4-50m^2+4)/(960m^3),$ \vspace*{2mm}
\item[(ii)] $S^2(m-1, 2m)=S^2(m+1, 2m)=(m^4+100m^2-6)/(1440m^3).$\vspace*{2mm}
\item[(iii)] $S^3(m+1, 4m)=(-m^4-180m^3+410m^2-4)/(7680m^3),$\vspace*{2mm}
\item[(iv)] $S^3(m-1, 4m)=(m^4-180m^3-410m^2+4)/(7680m^3),$\vspace*{2mm}
\item[(v)] $S^2(m-1, 4m)=S^2(m+1, 4m)=(m^4+820m^2-6)/(11520m^3).$
\end{itemize}
\end{lemma}

\section{Real quadratic fields with class number $3$}
In this section, we compute the value $\zeta_k(-1, 
\mathfrak{A})$ for some ideal class $\mathfrak{A}$ in $k$, and then compare these values to $\zeta_k(-1)$ to derive our results. Throughout this section, $k$ is a real quadratic field of Richaud-Degert (R-D) type, more precisely $k=\mathbb{Q}(\sqrt{d})$ with radicand $d = n^2 + r$ satisfying $r\mid 4n$ and $-n<r\leq n$.
Degert \cite[Satz 1]{D49} shows that the fundamental unit $\varepsilon$ of $k$ and its norm $N(\varepsilon)$ are:
\begin{align}
\varepsilon =\begin{cases} n+\sqrt{n^2+r}, & N(\varepsilon)=-\textsl{sgn } r, \hspace{3mm} \text{ if }  |r|=1,\\
\dfrac{n+\sqrt{n^2+r}}{2},&  N(\varepsilon)=-\textsl{sgn } r, \hspace{3mm} \text{ if } |r|=4,\\
\dfrac{2n^2+r}{|r|}+\dfrac{2n}{|r|}\sqrt{n^2+r}, & N(\varepsilon)=1, \hspace{13.5mm} \text{ if } |r|\neq1,4.\\
\end{cases}
\end{align}
It is easy to see that  $n^2+1\not\equiv 3\pmod 4,~ n^2-1\not\equiv 1, 2\pmod 4, ~n^2\pm 4\not\equiv 2,3 \pmod 4$ and $n^2-4\not\equiv 1 \pmod 8$. Thus to cover all real quadratic fields of narrow R-D type, it is enough to consider the following cases:
\begin{itemize}
\item[(i)] $n^2+1\equiv 1, 2 \pmod 4$,
\item[(ii)] $n^2-1\equiv 3\pmod 4$,
\item[(iii)] $n^2+4\equiv 1\pmod 4$, 
\item[(iv)] $n^2-4\equiv 5\pmod 8$.
\end{itemize}

We consider the real quadratic field $k=\mathbb{Q}(\sqrt{d})$ of R-D type with $d\equiv 1\pmod 8$.  Then $2$ splits in $k$, that is, 
$$(2)=\Big(2,\frac{1+\sqrt{d}}{2}\Big)\Big(2,\frac{1-\sqrt{d}}{2}\Big).$$ 
Note that $n$ is even if $|r|\ne 1,4$. In fact, when $n$ is odd, one
has that $1 \equiv d = n^2 + r\equiv r + 1 \pmod 8$, which is contrary to the assumption $r\mid 4n$. 
The case $|r| = 4$ can not occur since $n^2 \pm 4\not\equiv 1 \pmod 8$. We extract the following result from Theorem 2.3 of \cite{BK1}.
\begin{theorem}[Byeon and Kim \cite{BK1}]\label{thm3.1}
Let $d=n^2+r$, and let $k=\mathbb{Q}{(\sqrt{d})}$ be a real quadratic field of R-D type. Let $\mathfrak{P}$ be denote the ideal class of principal ideals of $k$. If $d \equiv 1 \pmod 8$, then 
\begin{equation*}
\zeta_k(-1, \mathfrak{P})=\begin{cases}
\dfrac{n^3+14n}{360},  &{\rm ~if~} |r|=1,\\
\\
\dfrac{2n^3(r^2+1)+n(3r^3+50r^2+3r)}{720r^2},  &{\rm ~if~} |r|\neq 1,4.
\end{cases}
\end{equation*}
\end{theorem}
The following result can be extracted from \cite[Theorem 2.5]{BK2}. However for the sake of completeness, we provide a detailed proof. 
\begin{theorem}[Byeon and Kim \cite{BK2}]\label{thm3.2}
Let $d=n^2+r$, and let $k=\mathbb{Q}{(\sqrt{d})}$ be a real quadratic field of R-D type. Let $\mathfrak{A}$ be the ideal class containing $\Big(2,(1+\sqrt{d})/2\Big)$ or $\Big(2,(1-\sqrt{d})/2\Big)$. If $d \equiv 1\pmod 8$, then
\begin{equation*}
\zeta_k(-1, \mathfrak{A})=\begin{cases}
\dfrac{n^3+104n}{1440},  &{\rm ~if~} |r|=1,\\
\\
\dfrac{2n^3(r^2+1)+n(3r^3+410r^2+3r)}{2880r^2},  &{\rm ~if~} |r|\neq 1,4.
\end{cases}
\end{equation*}
\end{theorem}

\begin{proof}
Let us assume that $\mathfrak{a}:=\left(2, (1+\sqrt{d})/2\right)\in \mathfrak{A}^{-1}$. Then $\{r_1=(1+\sqrt{d})/2, ~r_2=2\}$ is an integral basis for $\mathfrak{a}$ and thus $\delta(\mathfrak{a})=2\sqrt{d}$.
We will give computations in detail for $|r|=1$, and the similar argument goes through for other cases. By Lemma \ref{2.1}, we get 
$$\begin{bmatrix}
a&b\\
c&d
\end{bmatrix}
=\begin{bmatrix}
n+1&(d-1)/4\\
4&n-1
\end{bmatrix}
.$$
Since $n^2+1 \equiv 1 \pmod 8$, so that $4|n$, and thus $n\pm 1 \equiv\pm 1\pmod 4$. Hence by Lemma \ref{DS1}, we obtain \vspace*{2mm}\\
$240c^3(\textsl{sgn }c)S^3(a,c)=240c^3S^3(n+1,4)=240\times4^3S^3(1,4)=-360,$\vspace*{2mm}
$240c^3(\textsl{sgn } c)S^3(d,c)=240c^3S^3(n-1,4)=240\times4^3S^3(-1,4)=360,$
$$180ac^3(\textsl{sgn }c)S^2(a,c)=180ac^3S^2(n+1,4)=180\times 4^3a S^2(1,4)=410(n+1),$$
$$180dc^3(\textsl{sgn }c)S^2(d,c)=180dc^3S^2(n-1,4)=180\times 4^3dS^2(-1,4)=410(n-1).$$
By Theorem \ref{thm2.2}, we get $$\zeta_k(-1, \mathfrak{A})=\frac{n^3+104n}{1440}.$$
\end{proof}
Let $h(d)$ denote the class number of $\mathbb{Q}(\sqrt{ d})$.
\begin{theorem}\label{thm3.3}
Let $d=n^2+1\equiv 1\pmod 8$ be a square-free integer. If $h(d)=3$ then $$\sum_{\substack{ |t|<\sqrt{d}\\ t^2\equiv d\pmod 4}}\sigma\left(\frac{d-t^2}{4}\right)=\frac{n^3+44n}{4}.$$
The converse holds if $h(d)$ is odd with one exception, viz. $d=17$. 
\end{theorem}

\begin{proof}
Let us assume that the class group of $k=\mathbb{Q}(\sqrt{d})$ is $\mathfrak{C}(k)=\{\mathfrak{P}, \mathfrak{A}, \mathfrak{B}\}$ with principal ideals class $\mathfrak{P}$. Then by Theorem \ref{thm3.1}, we have  $$\zeta_k(-1, \mathfrak{P})=\dfrac{n^3+14n}{360}.$$
 If $\Big(2,(1\pm\sqrt{d})/2\Big) \in \mathfrak{A}^{-1}=\mathfrak{B}$, then by Theorem \ref{thm3.2}, we see that $\zeta_{k}(-1,\mathfrak{P})=\zeta_{k}(-1,\mathfrak{A})$ if and only if $d=17$. Thus $\Big(2,(1-\sqrt{d})/2\Big)$ and $\Big(2,(1+\sqrt{d})/2\Big)$ are non-principal ideals except $d=17$. 

Let $\Big(2,(1-\sqrt{d})/2\Big)\in \mathfrak{A}$ and $\Big(2,(1+\sqrt{d})/2\Big)\in \mathfrak{B}$. Then by Theorem \ref{thm3.2}, we obtain 
$$\zeta_k(-1, \mathfrak{A})=\zeta_k(-1, \mathfrak{B})=\dfrac{n^3+104n}{1440}.$$ 
As $\mathfrak{C}(k)=\{\mathfrak{P}, \mathfrak{A}, \mathfrak{B}\}$,  we obtain $$\zeta_{k}(-1)=\zeta_{k}(-1,\mathfrak{P})+\zeta_{k}(-1,\mathfrak{A})+\zeta_{k}(-1,\mathfrak{B})=\frac{n^3+44n}{240}.$$
We now apply Theorem \ref{thm2.1} to get 
$$\sum_{\substack{ |t|<\sqrt{d}\\ t^2\equiv d\pmod 4}}\sigma\left(\dfrac{d-t^2}{4}\right)=\dfrac{n^3+44n}{4}.$$

Converse part implies 
\begin{equation}\label{z1}
\zeta_{k}(-1)=\dfrac{n^3+44n}{240}.
\end{equation}
Then by \cite[Theorem 2.4]{BK1} and \cite[Theorem 2.7]{BK2}, we obtain $h(d)\geq3$. If $h(d)>3$, then there exist at least $5$ ideal classes in $k$ since $h(d)$ is odd. If $\mathfrak{C}$ and $\mathfrak{D}$ are another two ideal classes in $k$, then  
\begin{equation}\label{eq3.2}
\zeta_{k}(-1)\geq\zeta_{k}(-1,\mathfrak{P})+\zeta_{k}(-1,\mathfrak{A})+\zeta_{k}(-1,\mathfrak{B})+\zeta_{k}(-1,\mathfrak{C})+\zeta_{k}(-1,\mathfrak{D}),
\end{equation}
where the equality holds if $h(d)=5$. Without loss of generality, let us assume that  
$\Big(2,(1-\sqrt{d})/2\Big)\in \mathfrak{A}$ and $\Big(2,(1+\sqrt{d})/2\Big)\in \mathfrak{B}$. Then by Theorem \ref{thm3.2}, we obtain 
$$\zeta_k(-1, \mathfrak{A})=\zeta_k(-1, \mathfrak{B})=\dfrac{n^3+104n}{1440}.$$ 
Since for any ideal class $\mathfrak{Q}$, $\zeta_{k}(-1,\mathfrak{Q})>0$, thus by \eqref{eq3.2} we obtain 
$$\zeta_{k}(-1)>\zeta_{k}(-1,\mathfrak{P})+\zeta_{k}(-1,\mathfrak{A})+\zeta_{k}(-1,\mathfrak{B})=\frac{n^3+44n}{240},$$ which contradicts to \eqref{z1}. This completes the proof.
\end{proof}
We can prove the following result using a similar argument as in Theorem \ref{thm3.3}, and by using Lemma \ref{DS2}.
\begin{theorem}\label{thm3.4}
Let $n$ be a positive integer and $d=n^2+1\equiv 2\pmod 4$  be  square-free. Let $p$ be a prime divisor of  $n$. If $h(d)=3$, then
\begin{equation*}
\sum_{\substack{ |t|<\sqrt{d}\\ t^2\equiv d\pmod 4}}\sigma\left(\frac{d-t^2}{4}\right)=\frac{2n^3+13n}{3}+ \frac{8n^3+2n(p^4+10p^2)}{3p^2}.
\end{equation*}
The converse holds if $h(d)$ is odd with one exception, viz. $d = 2$.
\end{theorem}
We now consider a family of real quadratic fields of wide R-D type, and deduce the similar criteria for class number $3$. The proof of the following result goes along the similar lines of Theorem \ref{thm3.3}.
\begin{theorem}\label{thm3.5}
Let $d=n^2+r \equiv 1\pmod 8$ be a square-free integer with $|r|\neq 1,4$.  If $h(d)=3$, then 
$$\sum_{\substack{ |t|<\sqrt{d}\\ t^2\equiv d\pmod 4}}\sigma\left(\frac{d-t^2}{4}\right)=\frac{2n^3(r^2+1)+n(3r^3+170r^2+3r)}{8r^2}.
$$
The converse holds if $h(d)$ is odd with one exception, viz. $d=33$.
\end{theorem}

\section{Real quadratic fields with class numbers $2$ and $3$}
In this section, we obtain class number $2$ and $3$ criteria for the real quadratic fields $\mathbb{Q}(\sqrt{n^2+r})$ when $r\in\{1,4\}$ and $n^2+r\equiv 5\pmod 8$. These two families were not considered in \cite{BK2} to obtain the class number $2$ criteria. 

Let $d=n^2+4 \equiv 5\pmod 8$, and let $p$ be an odd prime divisor of  $n$. Then
$p$ splits in $k=\mathbb{Q}(\sqrt{d})$, that is
$
(p)=\mathfrak{p}\mathfrak{p}'$, where $\mathfrak{p}=\left(p,(p+2+ \sqrt{d})/2\right)$ and $\mathfrak{p}'=\left(p,(p+2-\sqrt{d})/2\right).
$
Simillarly if $d=n^2+1 \equiv 5\pmod 8$, and $q$ is an odd prime divisor of $n$. Then $q$ also splits in $k=\mathbb{Q}(\sqrt{d})$, that is
$(q)=\mathfrak{q}\mathfrak{q}'$, where $\mathfrak{q}=\left(q,(1+\sqrt{d})/2\right)$ and $\mathfrak{q}'=\left(q,(1-\sqrt{d})/2\right).$


We can prove the following result using the similar argument of the proof of Theorem \ref{thm3.2}.

\begin{theorem}\label{thm4.1}
Let $d=n^2+r\equiv 5\pmod 8$ be square-free with $r=1, 4$, and let $k=\mathbb{Q}{(\sqrt{d})}$. Let $p$ be an odd prime divisor of $n$. If  $\mathfrak{A}$  is the ideal class containing one of $\mathfrak{p}$,  $\mathfrak{p}'$, $\mathfrak{q}$ and $\mathfrak{q}'$ (as defined above), then
\begin{equation*}
\zeta_k(-1, \mathfrak{A})=\begin{cases}
(n^3+n(p^4+10p^2))/(360p^2),  &{\rm ~if~} r=4,\\
(n^3+n(4q^4+10q^2))/(360q^2),  &{\rm ~if~} r=1.
\end{cases}
\end{equation*}
\end{theorem}
Let $\mathfrak{P}$ be the ideal class of principal ideals in $k$. Then 
\begin{equation*}
\zeta_k(-1, \mathfrak{P})=\begin{cases}
(n^3+11n)/360,  &{\rm ~if~} r=4,\\
(n^3+14n)/360,  &{\rm ~if~} r=1.
\end{cases}
\end{equation*}
Thus $h(d)>1$ if $\zeta_k(-1, \mathfrak{P})\ne\zeta_k(-1, \mathfrak{A})$. On the other hand, $\zeta_k(-1, \mathfrak{P})=\zeta_k(-1, \mathfrak{A})$ implies $$n=\begin{cases}
p,  &{\rm ~if~} r=4,\\
2q,  &{\rm ~if~} r=1.
\end{cases}$$  
\begin{remark}\label{rk4.1}
Let $d$ be as in Theorem \ref{thm4.1}. If $h(d)=1$, then $d$ must be of the form either $p^2+4$ or $4p^2+1$.
\end{remark}
This remark does not provide any information about the conjectures (C) and (Y).
One can prove the following result using a similar argument of the proof of Theorem \ref{thm3.3}.
\begin{theorem}\label{thm4.2}
Let $k$ and $p$ be as in Theorem \ref{thm4.1}. If $h(d)=3$ then
\begin{equation*}
\sum_{\substack{ |t|<\sqrt{d}\\ t^2\equiv d\pmod 4}}\sigma\left(\frac{d-t^2}{4}\right)=\begin{cases}
\frac{n^3+11n}{6}+ \frac{n^3+n(p^4+10p^2)}{3p^2},  &{\rm ~if~} r=4 {\rm ~and~} n\neq p ,\\
\frac{n^3+14n}{6}+\frac{n^3+n(4p^4+10p^2)}{3p^2},  &{\rm ~if~} r=1 {\rm ~and~} n\neq 2p .
\end{cases}
\end{equation*}
The converse holds if $h(d)$ is odd.
\end{theorem}
Along the same line, we obtain the following criteria for class number $2$.
\begin{theorem}\label{thm4.3}
Let $k$ and $p$ be as in Theorem \ref{thm4.1}. Then $h(d)=2$ if and only if
\begin{equation*}
\sum_{\substack{ |t|<\sqrt{d}\\ t^2\equiv d\pmod 4}}\sigma\left(\frac{d-t^2}{4}\right)=\begin{cases}
\frac{n^3+11n}{6}+ \frac{n^3+n(p^4+10p^2)}{6p^2},  &{\rm ~if~} r=4 {\rm ~and~} n\neq p ,\\
\frac{n^3+14n}{6}+\frac{n^3+n(4p^4+10p^2)}{6p^2},  &{\rm ~if~} r=1 {\rm ~and~} n\neq 2p .
\end{cases}
\end{equation*}
\end{theorem}
Note that Byeon and Lee \cite{BL08} proved that if $d=n^2+1$ is even square-free integer with $d>362$, then $h(d)\geq 3$. In particular, they proved that $d=10, 26, 122, 362$ are the only values of $d$ for which $h(d)=2$.
\section{Computations and concluding remarks}
In this section, we give some numerical examples which verify our results in \S3 and \S4. We use SAGE version $8.4$ (2018-10-17) for all the computations in this paper. We have computed $h(d)$ and verified Theorem \ref{thm3.3} for $d\leq 10^{10}$ when $d$ is composite,  and $d\leq 10^{13}$ when $d$ is prime. We have obtained only one $d$, viz. $d=257$, with $h(d)=3$ under the assumptions of this theorem. 

In case of Theorem \ref{thm3.5}, we have computed $h(d)$ for $n\leq 10^{4}$ and $|r|\leq 4\times 10^{4}$. Out of these, we have obtained only $2$ fields with $h(d)=3$. These values are listed in Table \ref{t0}, and Theorem \ref{thm3.5} is verified for all these values. 
\begin{table}[ht]
 \centering
\begin{tabular}{  c  c  c  c c c c c } 
 \hline
 $n$ & $r$ & $d$ & $h(d)$\\
\hline
18&$-3$&321&3\\
\hline
22&$-11$&473&3\\
\hline
\end{tabular}\vspace{2mm}
\caption{Numerical examples of Theorem \ref{thm3.5}. }\label{t0}
\end{table}

We have computed $h(d)$ for $d\leq 10^8$ satisfying the conditions in Theorem \ref{thm4.2}. We have listed in Table \ref{t1} the only those values  which correspond to $h(d)=3$. We have verified the equation in Theorem \ref{thm4.2} by computation for the values listed in Table \ref{t1}. There are only $5$ real quadratic fields of the form $\mathbb{Q}(\sqrt{n^2+r})$ satisfying $n^2+r\equiv 5\pmod 8$ with $r=1,4$ and $n^2+r\leq 10^8$. Out of these fields, $1$ field is of the form $\mathbb{Q}(\sqrt{n^2+1})$ and $3$ fields are of the other form.
\begin{table}[ht]
 \centering
\begin{tabular}{  c  c  c  c c c c c } 
 \hline
 $n$ & $r$ & $d$ & $p$ & $h(d)$\\
\hline
54&1&2917&3&3\\
\hline
15&4&229&3&3\\
\hline
15&4&229&5&3\\
\hline
27&4&733&3&3\\
\hline
35&4&1229&5&3\\
\hline
35&4&1229&7&3\\
\hline
\end{tabular}\vspace{2mm}
\caption{Numerical examples of Theorem \ref{thm4.2}. }\label{t1}
\end{table}

Similary, we have computed $h(d)$ for $d\leq 10^{10}$ satisfying the conditions in Theorem \ref{thm4.3}. There are only $3$ real quadratic fields of the form $\mathbb{Q}(\sqrt{n^2+r})$ satisfying $n^2+r\equiv 5\pmod 8$ with $r=1,4, ~ h(d)=2$ and $n^2+r\leq 10^{10}$. Out of these fields, no field is of the form $\mathbb{Q}(\sqrt{n^2+1})$ and $2$ fields are of the other form. We have listed these values in Table \ref{t2}, and verified the equation in Theorem \ref{thm4.2} for them.
\begin{table}[ht]
 \centering
\begin{tabular}{  c  c  c  c c c c c } 
 \hline
 $n$ & $r$ & $d$ & $p$ & $h(d)$\\
\hline
9&4&85&3&2\\
\hline
25&4&629&5&2\\
\hline
\end{tabular}\vspace{2mm}
\caption{Numerical examples of Theorem \ref{thm4.3}. }\label{t2}
\end{table}
\section*{Acknowledgements} 
\noindent The authors are indebted to the referee for going through the manuscript very carefully and suggesting many changes which has helped improving the manuscript. The authors are thankful to Mr. Rajiv Dixit for his kind help in computations. A. Hoque is supported by the SERB-NPDF (PDF/2017/001958), Govt. of India. M. Mishra is partially supported by Infosys grant.

\end{document}